\documentclass{amsart}

\usepackage{latexsym,amssymb,amsmath}

\theoremstyle{plain}
\newtheorem{lemma}{Lemma}[section]
\newtheorem{theorem}[lemma]{Theorem}
\newtheorem{conjecture}[lemma]{Conjecture}

\newtheorem{proposition}[lemma]{Proposition}

\theoremstyle{definition}
\newtheorem{example}[lemma]{Example}
\newtheorem{remark}[lemma]{Remark}
\newtheorem{definition}[lemma]{Definition}

\title[]{On two classes of regular sequences} 
\author[ ]{Ri-Xiang Chen} 
\address{Department of Mathematics, Nanjing University of Science and Technology, Nanjing, Jiangsu, 210094, P.R.China}
\email{rc429@cornell.edu}
\thanks{This work was partially supported by National Natural Science Foundation of China (Grant No. 11326064)}

\date{January 26,2015}

\keywords{regular sequence, power sum, matrix, distance.}
\subjclass[2000]{ Primary 05E05, 13A02; Secondary 15A03, 54A10.}

\begin{document}
\begin{abstract}
In this paper we get two new classes of regular sequences 
in the polynomial ring over the field of complex numbers. 
\end{abstract}
\maketitle

\section{Introduction}
\noindent Throughout this paper $S=k[x_1,x_2,\ldots,x_n]$ denotes the polynomial ring 
in $n$ variables over a field $k$. Let $f_1,f_2,\ldots,f_n$ be some homogeneous 
polynomials in $S$ of degrees $a_1,a_2,\ldots,a_n$ with $1\leq a_1 \leq a_2 \leq \cdots \leq a_n$.
It is well-known that the following statements are equivalent:
\begin{itemize}
\item[(1)] $f_1,f_2,\ldots,f_n$ is a regular sequence in $S$;
\item[(2)] The Koszul complex $K.(f_1,f_2,\ldots,f_n)$ is the minimal free resolution of $S/(f_1,f_2,\ldots,f_n)$ over $S$;
\item[(3)] $(f_1,f_2,\ldots,f_n)_N=S_N$ where $N=1+\sum_{i=1}^{n}(a_i-1)$;
\item[(4)] $\sqrt{(f_1,f_2,\ldots,f_n)}=(x_1,x_2,\ldots,x_n)$, that is, $\mbox{ht}(f_1,f_2,\ldots,f_n)=n$;
\end{itemize}
and when the field $k$ is algebraically closed, they are all equivalent to the following:
\begin{itemize}
\item[(5)] $V(f_1,f_2,\ldots,f_n)=\{(0,0,\ldots,0)\}$, that is, $(0,0,\ldots,0)$ is 
the only solution to the system of equations $f_1=0, f_2=0, \ldots, f_n=0$.
\end{itemize}    
There are some other criterions for $f_1,f_2,\ldots,f_n$ being a regular sequence. 
Despite the many criterions, there remain some 
very difficult questions about regular sequences.  

For example, in \cite{B:CKW}, Conca et al. made 
the following conjecture:

\begin{conjecture} \cite{B:CKW} \label{L:Cc}
In $S=\mathbb{C}[x_1,x_2,x_3]$, let $p_a(3)=x_1^a+x_2^a+x_3^a$,
$p_b(3)=x_1^b+x_2^b+x_3^b$, $p_c(3)=x_1^c+x_2^c+x_3^c$, 
where $a,b,c\in \mathbb{N}^*$, $a<b<c$,
$\mbox{gcd}(a,b,c)=1$ and $6|abc$, then $p_a(3),p_b(3),p_c(3)$ 
form a regular sequence in $S$.
\end{conjecture}

Conca et al. verified some special cases of this 
conjecture and the conjecture remains mysterious.
Another conjecture about regular sequences 
is the famous EGH Conjecture, which is wide open except some special cases.

\begin{conjecture}(Eisenbud-Green-Harris) \cite{B:EGH} \label{C:EGH}
If $I\subset S=k[x_1,x_2,\ldots,x_n]$ is 
a homogeneous ideal containing a regular sequence of 
homogeneous polynomials $f_1, f_2, \ldots, f_n$ 
of degrees $2\leq a_1\leq a_2\leq \cdots \leq a_n$ 
, then there exists a homogeneous ideal in $S$ containing 
$x_1^{a_1}, x_2^{a_2}, \ldots, x_n^{a_n}$ with the same Hilbert function. 
\end{conjecture}

These conjectures motivate us to further our study of 
regular sequences. 
It is well-known that generically, $f_1,f_2,\ldots,f_n$ 
form a regular sequence in $S$(see Proposition \ref{P:generic}), but we do not have many 
specific examples of regular sequences. 
The most popular example is the sequence $x_1^{a_1},x_2^{a_2},\ldots,x_n^{a_n}$.
Also, if under a monomial order $<$, $\mbox{in}_{<}(f_1)=x_1^{a_1}$,
$\mbox{in}_{<}(f_2)=x_2^{a_2}$,$\ldots$,$\mbox{in}_{<}(f_n)=x_n^{a_n}$,
then $f_1,f_2,\ldots,f_n$ form a regular sequence in $S$.
In addition, in the case of $a_1=a_2=\cdots=a_n=1$, $f_1,f_2,\ldots,f_n$ 
is a regular sequence if and only if $f_1,f_2,\ldots,f_n$ 
are $k$-linearly independent, and equivalently, the 
determinant of the $n\times n$ coefficient matrix is nonzero.

It is easy to see that if $f_1=g_1h_1$ with $\mbox{deg}(g_1)\geq 1$ 
and $\mbox{deg}(h_1)\geq 1$, then $f_1,f_2,\ldots,f_n$ is a regular sequence 
if and only if both $g_1,f_2,\ldots,f_n$ and $h_1,f_2,\ldots,f_n$ are regular sequences.
So, in \cite{B:Ab}, Abedelfatah considered the following class 
of regular sequences. Let $f_1,f_2,\ldots,f_n$ be homogeneous 
polynomials which split into linear factors, then $f_1,f_2,\ldots,f_n$
is a regular sequence if and only if for $i=1,2,\ldots,n$ and for every 
linear factor $l_i$ of $f_i$, $l_1,l_2,\ldots,l_n$ is a regular sequence.
Abedelfatah proved that the EGH Conjecture holds for this class of 
regular sequences.

The goal of this paper is to get two new classes of regular sequences
in the polynomial ring over the field of complex numbers.

In Section 2 we consider the power sum symmetric polynomial 
$p_m(n)=x_1^m+x_2^m+\cdots+x_n^m$ in $S=\mathbb{C}[x_1,x_2,\ldots,x_n]$. 
We prove that if $a,d\in \mathbb{N}^*$ with $\mbox{gcd}(a,d)=1$ and 
$n!|a(a+d)\cdots (a+(n-1)d)$, then 
$p_a(n),p_{a+d}(n),\ldots, p_{a+(n-1)d}(n)$ is a regular sequence in $S$. (See Theorem \ref{T:pst}.)

In Section 3 we prove that if $f_1,f_2,\ldots,f_n$ are homogeneous 
polynomials of degrees $a_1,a_2,\ldots,a_n$ in $S=\mathbb{C}[x_1,x_2,\ldots,x_n]$
and for all $i=1,2,\ldots,n$, the distance between $f_i$ and $x_i^{a_i}$ (see Definition \ref{D:distance}) 
is less than $1$, then  $f_1,f_2,\ldots,f_n$ is a regular sequence in $S$. (See Theorem \ref{T:close}.)

\textbf{Acknowledgments}. The author would like to thank the referee for 
his/her valuable comments.

\section{Regular sequences of power sums}

\noindent Throughout this section $S=\mathbb{C}[x_1,x_2,\ldots,x_n]$ and let
$p_m(n)$ denote the power sum symmetric polynomial $x_1^m+x_2^m+\cdots+x_n^m$
in $S$. Given positive integers $a_1<a_2<\cdots<a_n$, let $A$ be 
the set $\{a_1,a_2,\ldots,a_n\}$. For simplicity, we will 
denote the sequence $p_{a_1}(n),p_{a_2}(n),\ldots,p_{a_n}(n)$ in $S$
by $p_{A}(n)$.

The question is: when is $p_{A}(n)$  a regular sequence in $S$? 
To answer it, Conca et al. made the following two observations.

\begin{lemma}\cite{B:CKW}\label{L:Cl1}
Let $a_1<a_2<\cdots<a_n$ be positive integers and $A=\{a_1,a_2,\ldots,a_n\}$.
Set $d=\mbox{gcd}(a_1,a_2,\ldots,a_n)$ and $A'=\{a/d|a\in A\}$.
Then $p_A(n)$ is a regular sequence in $S$ if and only if 
$p_{A'}(n)$ is a regular sequence in $S$.
\end{lemma}

\begin{lemma}\cite{B:CKW}\label{L:Cl2}
Let $f_1,f_2,\ldots,f_n$ be a regular sequence of homogeneous 
symmetric polynomials of degrees $1\leq a_1 \leq a_2 \leq \cdots \leq a_n$
in $S=\mathbb{C}[x_1,x_2,\ldots,x_n]$, then $n!$ divides $a_1a_2\cdots a_n$,
that is, $n!|a_1a_2\cdots a_n$.
\end{lemma}

By Lemma \ref{L:Cl1} we can always assume that $\mbox{gcd}(a_1,a_2,\ldots,a_n)=1$.
And Lemma \ref{L:Cl2} implies that $n!|a_1a_2\cdots a_n$ is a necessary 
condition for $p_A(n)$ to be a regular sequence in $S$. One may wonder if 
both $\mbox{gcd}(a_1,a_2,\ldots,a_n)=1$ and $n!|a_1a_2\cdots a_n$ would imply
$p_A(n)$ being a regular sequence. This is true for $n=2$, which is proved 
in \cite{B:CKW}, and is also a special case of the following Theorem \ref{T:pst}. 
The case of $n=3$ remains mysterious as Conjecture \ref{L:Cc}.
However, for $n\geq 4$, the above two conditions are not sufficient for 
$p_A(n)$ being a regular sequence, as is illustrated in the following example.

\begin{example}
$\mbox{gcd}(1,3,5,24)=1$ and $4!|1\times 3\times 5\times 24$, but 
$p_1(4),p_3(4),p_5(4),p_{24}(4)$ is not a regular sequence in $S=\mathbb{C}[x_1,x_2,x_3,x_4]$,
because $(e^{\frac{\pi i}{24}},-e^{\frac{\pi i}{24}},1,-1)$ is a nonzero solution
to the system of equations $p_1(4)=0,p_3(4)=0,p_5(4)=0,p_{24}(4)=0$. 
Examples can be similarly constructed for $n>4$.
\end{example}

In \cite{B:CKW} Proposition 2.9, Conca et al. proved that if $a_1,a_2,\ldots,a_n$
are consecutive positive integers then $p_A(n)$ is a regular sequence in $S$. 
To generalize this result we have the following theorem.

\begin{theorem}\label{T:pst}
Let $S=\mathbb{C}[x_1,x_2,\ldots,x_n]$, $n\geq 2$, 
$a,d\in \mathbb{N}^*$ and $A=\{a,a+d,a+2d,\ldots,a+(n-1)d\}$.
Assume that $\mbox{gcd}(a,d)=1$ and $n!|a(a+d)(a+2d)\cdots (a+(n-1)d)$,
then $p_A(n)$ is a regular sequence in $S$.
\end{theorem}

Note that the assumption $\mbox{gcd}(a,d)=1$ is equivalent 
to $\mbox{gcd}(a,a+d,a+2d,\ldots,a+(n-1)d)=1$. And when $d=1$,
this theorem is Proposition 2.9 in \cite{B:CKW}.
To prove Theorem \ref{T:pst}, we need to prove the following 
two lemmas, which themselves are interesting.

\begin{lemma}\label{L:nt}
Let $n,a,d\in \mathbb{N}^*$, $n\geq 2$ and $\mbox{gcd}(a,d)=1$.
Then $n!|a(a+d)(a+2d)\cdots(a+(n-1)d)$ if and only if 
$\mbox{gcd}(d,n!)=1$, that is, either $d=1$ or every 
prime factor of $d$ is greater than $n$.
\end{lemma}

\begin{proof}
`only if': Suppose, for contradiction, that there is a prime number $p$ which 
divides $\mbox{gcd}(d,n!)$, then $p|d$ and $p|n!$.
Since $n!|a(a+d)(a+2d)\cdots(a+(n-1)d)$, it follows that 
$p|a+id$ for some $0\leq i\leq n-1$, which implies
that $p$ is a prime factor of $\mbox{gcd}(d,a+id)=\mbox{gcd}(d,a)$.
But $\mbox{gcd}(a,d)=1$, so we get a contradiction and therefore,
$\mbox{gcd}(d,n!)=1$.

`if':Assume $\mbox{gcd}(d,n!)=1$, then there exist integers
$s,t$ such that $sd+t(n!)=1$. Hence,
\begin{align*}
 &\ s^na(a+d)(a+2d)\cdots (a+(n-1)d) \\
 &\equiv  sa(sa+sd)(sa+2sd)\cdots (sa+(n-1)sd) \\
 &\equiv  sa(sa+1-t(n!))(sa+2-2t(n!))\cdots (sa+(n-1)-(n-1)t(n!))  \\
 &\equiv  sa(sa+1)(sa+2)\cdots (sa+(n-1)) \\
 &\equiv  0\quad \quad (\mbox{mod}\ n!) . 
\end{align*}
Since $\mbox{gcd}(s,n!)=1$, it follows that $n!|a(a+d)(a+2d)\cdots(a+(n-1)d)$.
\end{proof} 

In \cite{B:LL} Lam and Leung proved that if 
$z_1,z_2,\ldots,z_n\in \mathbb{C}$ are some $d$-th roots of unity
and $z_1+z_2+\cdots+z_n=0$, then $n$ must be 
a linear combination with non-negative integer coefficients of the 
prime facotrs of $d$. This is the main theorem of \cite{B:LL} and it 
was proved by group ring techniques. 
The next lemma can be viewed as a corollary to 
Lam and Leung's Theorem. For the completeness of 
the paper, we also give an elementary proof of the lemma. 

\begin{lemma}\label{L:root}
Let $d,n\in \mathbb{N}^*$, $d\geq 2$, $n\geq 2$ and 
$z_1,z_2,\ldots,z_n\in \mathbb{C}$ be some $d$-th roots of unity.
If $\mbox{gcd}(d,n!)=1$, then
$z_1+z_2+\cdots+z_n\neq 0$.
\end{lemma}

\begin{proof}
Let $d=q_1^{e_1}q_2^{e_2}\cdots q_s^{e_s}$ 
be the prime factorization of $d$, where $q_1, q_2, \ldots, q_s$ are 
distinct prime numbers and $e_1,e_2,\ldots, e_s\geq 1$ 
. Since $\mbox{gcd}(d,n!)=1$, 
it follows that $q_1> n,q_2> n,\ldots,q_s> n$.
We will prove this lemma by induction on $s$.

First we consider the case $s=1$. Suppose, for contradiction, 
that $z_1+z_2+\cdots+z_n= 0$. Let $\omega=e^{\frac{2\pi i}{d}}$.
Without the loss of generality, we can assume that 
$z_1=\omega^{b_1},z_2=\omega^{b_2},\ldots,z_n=\omega^{b_n}$ 
with $0=b_1\leq b_2\leq \cdots \leq b_n \leq d-1$. Then
\[
(b_2-b_1)+(b_3-b_2)+\cdots+(b_n-b_{n-1})+(b_1+d-b_n)=d.
\]
Since $q_1>n$, we have $q_1^{e_1}>nq_1^{e_1-1}$, so that
$d>n(d-\varphi(d))$, where $\varphi(d)=q_1^{e_1}-q_1^{e_1-1}$ 
is Euler's totient function. Hence, one of 
$b_2-b_1$,$b_3-b_2$, \ldots, $b_n-b_{n-1}$ and $b_1+d-b_n=d-b_n$ 
is greater than $d-\varphi(d)$.

Suppose $d-b_n>d-\varphi(d)$, then $b_n<\varphi(d)$.
Since $z_1+z_2+\cdots+z_n=\omega^{b_1}+\omega^{b_2}+\cdots+\omega^{b_n}=0$,
it follows that $\omega$ is a root of the polynomial $x^{b_1}+x^{b_2}+\cdots+x^{b_n}$,
which is of degree $b_n<\varphi(d)$. But this contradicts to the well-known fact that
the minimal polynomial of $\omega$ over $\mathbb{Q}$ has degree $\varphi(d)$.

Suppose $b_i-b_{i-1}>d-\varphi(d)$ for some $2\leq i\leq n$.
Since $z_1+z_2+\cdots+z_n=0$, it follows that
\begin{align*}
\frac{1}{z_i}(z_1+z_2+\cdots+z_n)&= \frac{1}{z_i}(z_i+z_{i+1}+\cdots+z_n+z_1+\cdots+z_{i-1})\\
                                &= 1+\omega^{b_{i+1}-b{i}}+\cdots+\omega^{b_n-b_i}+\omega^{b_1+d-b_i}+\cdots+\omega^{b_{i-1}+d-b_i}\\
                                &=0.
\end{align*}
Hence, $\omega$ is a root of the polynomial $1+x^{b_{i+1}-b{i}}+\cdots+x^{b_n-b_i}+x^{d-b_i}+\cdots+x^{b_{i-1}+d-b_i}$,
which is of degree $d-(b_i-b_{i-1})<d-(d-\varphi(d))=\varphi(d)$.
But this also contradicts to the fact that the minimal 
polynomial of $\omega$ over $\mathbb{Q}$ has degree $\varphi(d)$. 
So we have proved that $z_1+z_2+\cdots+z_n\neq 0$ for the case $s=1$.

Now we consider the case $s\geq 2$. Suppose, for contradiction, 
that $z_1+z_2+\cdots+z_n= 0$. Let $d_1=q_1^{e_1}$ and $d_2=q_2^{e_2}\cdots q_s^{e_s}$,
then $d=d_1d_2$ and $\mbox{gcd}(d_1,d_2)=1$.
Let $\omega=e^{\frac{2\pi i}{d_1}}$, $\varepsilon=e^{\frac{2\pi i}{d_2}}$ 
and $\eta=e^{\frac{2\pi i}{d}}$, then it is well-known that 
$\mathbb{Q}(\eta)=\mathbb{Q}(\omega)(\varepsilon)=\mathbb{Q}(\varepsilon)(\omega)$.
Hence, without the loss of generality, we can assume that
$z_1=\omega^{b_1}\varepsilon^{c_1}$, $z_2=\omega^{b_2}\varepsilon^{c_2}$, 
$\ldots$, $z_n=\omega^{b_n}\varepsilon^{c_n}$ with
$0=b_1\leq b_2\leq \cdots \leq b_n\leq d_1-1$ and $0=c_1\leq c_2\leq \cdots \leq c_n\leq d_2-1$.
Then
\[
(b_2-b_1)+(b_3-b_2)+\cdots+(b_n-b_{n-1})+(b_1+d_1-b_n)=d_1.
\]
By the proof of the case $s=1$, it is easy to see that we can assume
$d_1-b_n\geq \max\{b_i-b_{i-1}|i=2,\ldots,n\}$.
Since $q_1>n$, we have $q_1^{e_1}>nq_1^{e_1-1}$, so that
$d_1>n(d_1-\varphi(d_1))$, which implies $d_1-b_n>d_1-\varphi(d_1)$, 
that is, $b_n<\varphi(d_1)$. 

Let $j=\min\{i=1,\ldots,n|b_i=b_n\}\geq 1$, then 
$\varepsilon^{c_j},\varepsilon^{c_{j+1}},\ldots, \varepsilon^{c_n}$
are some $d_2$-th roots of unity. Since $\mbox{gcd}(d,n!)=1$, it 
follows that $\mbox{gcd}(d_2,(n-j+1)!)=1$. Thus, by the induction hypothesis,
it is easy to see that  $\varepsilon^{c_j}+\varepsilon^{c_{j+1}}+\cdots+\varepsilon^{c_n}\neq 0$.
Therefore, $\omega$ is a root of the polynomial 
$\varepsilon^{c_1}x^{b_1}+\varepsilon^{c_2}x^{b_2}+\cdots+\varepsilon^{c_{j-1}}x^{b_{j-1}}
+(\varepsilon^{c_j}+\varepsilon^{c_{j+1}}+\cdots+\varepsilon^{c_n})x^{b_n}\in \mathbb{Q}(\varepsilon)[x]$,
which is of degree $b_n<\varphi(d_1)$. However, since $\varphi(d)=\varphi(d_1)\varphi(d_2)$ and
\begin{align*}
\varphi(d)&=[\mathbb{Q}(\eta):\mathbb{Q}]\\
          &=[\mathbb{Q}(\varepsilon)(\omega):\mathbb{Q}(\varepsilon)][\mathbb{Q}(\varepsilon):\mathbb{Q}]\\
          &=[\mathbb{Q}(\varepsilon)(\omega):\mathbb{Q}(\varepsilon)]\varphi(d_2),
\end{align*}
it follows that $[\mathbb{Q}(\varepsilon)(\omega):\mathbb{Q}(\varepsilon)]=\varphi(d_1)$, 
so that the minimal polynomial of $\omega$ over $\mathbb{Q}(\varepsilon)$ has degree $\varphi(d_1)$.
So we have a contradiction, and therefore, $z_1+z_2+\cdots+z_n\neq 0$.
\end{proof}

\begin{proof}[Proof of Theorem \ref{T:pst}]
Suppose, for contradiction, that $p_A(n)$ is not a regular sequence in $S$,
then the polynomial system associated to $p_A(n)$ has a nonzero solution
$(z_1,z_2,\ldots,z_n)\in \mathbb{C}^n$. Since the polynomials $p_A(n)$ are 
homogeneous and symmetric, without the loss of generality,
we can assume that $z_1=1$ and that $z_1^d,z_2^d,\ldots,z_t^d$ are all the 
distinct values among $z_1^d,z_2^d,\ldots,z_n^d$. Note that $1\leq t\leq n$.

If $t=1$ and $d=1$, then $z_1=z_2=\cdots=z_n=1$. Hence, $z_1^a+z_2^a+\cdots+z_n^a=n\neq 0$,
which contradicts to $(z_1,z_2,\ldots,z_n)$ being a root of $p_a(n)$.

If $t=1$ and $d\geq 2$, then $z_1^d=z_2^d=\cdots=z_n^d=1$. Hence, 
$z_1^a,z_2^a,\ldots,z_n^a$ are some $d$-th roots of unity. 
By Lemma \ref{L:nt} and Lemma \ref{L:root}, we have that
$z_1^a+z_2^a+\cdots+z_n^a\neq 0$, which also contradicts 
to $(z_1,z_2,\ldots,z_n)$ being a root of $p_a(n)$.

If $t\geq 2$, then for every $1\leq i\leq t$, we set
\[
A_i=\{j=1,2,\ldots,n|z_j^d=z_i^d\} \ \ \ \ \mbox{and} \ \ \ \  w_i=\sum_{j\in A_i}z_j^a.
\]
Since $(z_1,z_2,\ldots,z_n)$ is a solution of 
the polynomial system associated to $p_A(n)$, 
we have the following system of equations:
\[
\left\{\begin{array}{ll}
w_1+w_2+\cdots+w_t=0\\
w_1z_1^d+w_2z_2^d+\cdots+w_tz_t^d=0\\
\cdots \cdots \cdots \\
w_1z_1^{(n-1)d}+w_2z_2^{(n-1)d}+\cdots+w_tz_t^{(n-1)d}=0\end{array}\right.
\]
Rewriting the first $t$ equations in the matrix form, we get
\[
\begin{pmatrix} 
1 & 1 & \cdots & 1 \\ 
z_1^d & z_2^d & \cdots &z_t^d \\ 
\vdots &\vdots &  & \vdots\\ 
z_1^{(t-1)d}& z_2^{(t-1)d}& \cdots & z_t^{(t-1)d}
\end{pmatrix}
\begin{pmatrix} 
w_1\\ 
w_2 \\ 
\vdots\\ 
w_t
\end{pmatrix}
=0.
\]
By Lemma \ref{L:nt} and Lemma \ref{L:root}, it is easy to see 
that $w_1\neq 0$, so that $(w_1,w_2,\ldots,w_t)$ is a nonzero vetor.
Thus, we have that 
\[
\begin{vmatrix} 
1 & 1 & \cdots & 1 \\ 
z_1^d & z_2^d & \cdots &z_t^d \\ 
\vdots &\vdots &  & \vdots\\ 
z_1^{(t-1)d}& z_2^{(t-1)d}& \cdots & z_t^{(t-1)d}
\end{vmatrix}
=0,
\]
which implies,
\[
\prod_{1\leq i< j\leq t}(z_j^d-z_i^d)=0.
\]
This contradicts to the assumption that 
$z_1^d,z_2^d,\ldots,z_t^d$ are all distinct.

Therefore, $p_A(n)$ is a regular sequence in $S$.
\end{proof}

\begin{remark}
In \cite{B:CKW}, Conca et al. verified some special cases of 
Conjecture \ref{L:Cc}. Now Theorem \ref{T:pst} proves some new 
cases of this conjecture. For example, if $A=\{1,8,15\}$ 
or $A=\{2,7,12\}$, then $p_A(3)$ is a regular sequence 
in $S=\mathbb{C}[x_1,x_2,x_3]$. However,
Conjecture \ref{L:Cc} still remains mysterious. 
For example, if $A=\{2,5,9\}$ then it is hard to see 
if $p_A(3)$ is a regular sequence. Also, when $A=\{1,6,m\}$, the proof for 
$p_A(3)$ being a regular sequence took almost $8$ pages
in \cite{B:CKW}, so one might wonder if there is an easier way to prove it. 

\end{remark}

\section{Regular sequences close to $x_1^{a_1},x_2^{a_2},\ldots,x_n^{a_n}$}
\noindent Let $S=k[x_1,x_2,\ldots,x_n]$ be the polynomial ring in $n$ variables 
over an infinite field $k$. Let $f_1,f_2,\ldots,f_n$ be a sequence of \
homogeneous polynomials of degrees $1\leq a_1\leq a_2\leq \cdots \leq a_n$ in $S$.
Then it is well-known that generically, 
$f_1,f_2,\ldots,f_n$ is a regular sequence in $S$.
In the next proposition, we will give this basic fact
an elementary proof, which will also be useful when proving Theorem \ref{T:close}.

\begin{proposition}\label{P:generic}
Let $S=k[x_1,x_2,\ldots,x_n]$ with $k$ being any field.
Given $1\leq a_1 \leq a_2 \leq \cdots \leq a_n$, 
let $W=S_{a_1}\oplus S_{a_2} \oplus \cdots \oplus S_{a_n}$ 
be the affine $k$-space of dimension $m$, 
where $m={\textstyle \sum_{i=1}^n\binom{n+a_i-1}{a_i} }$. Let
\[
\Sigma=\{(f_1,f_2,\ldots,f_n)\in W|f_1,f_2,\ldots,f_n\ \mbox{is a regular sequence in }\ S\}.
\]
Then $\Sigma$ is a non-empty Zariski-open subset of $W$.
Furthermore, if $k$ is an infinite field, then $\Sigma$
is a dense open subset of $W$, and then, generically, a 
sequence of homogeneous polynomials $f_1,f_2,\ldots, f_n$ 
of degrees $a_1,a_2,\ldots,a_n$ is a regular sequence in $S$.
\end{proposition}

\begin{proof}
For any $d\geq 1$, the set of all monomials in $S_d$ ordered 
in the lexicographic order form a basis of the $k$-vector space
$S_d$. And $S_0=k$ has basis $\{1\}$.
These bases give rise to a basis of the $k$-vector space 
$W=S_{a_1}\oplus S_{a_2} \oplus \cdots \oplus S_{a_n}$. 
We call these bases the monomial bases of these $k$-vector spaces.
So, under the monomial basis, $W$ is isomorphic to $k^m$. 

Let $N={\textstyle \sum_{i=1}^n(a_i-1)}$, then the $n$-tuple 
$(f_1,f_2,\ldots,f_n)$ is in $\Sigma$ if and only if for the ideal 
$(f_1,f_2,\ldots,f_n)\subset S$, we have that 
$(f_1,f_2,\ldots,f_n)_N=S_N$, or equivalently, 
given any $h\in S_N$, there exist 
$g_1\in S_{N-a_1}, g_2\in S_{N-a_2}, \ldots, g_n\in S_{N-a_n}$ 
such that 
\[
f_1g_1+f_2g_2+\cdots +f_ng_n=h. \tag{1}
\]

Under the monomial basis, let $\lambda_1,\lambda_2,\ldots,\lambda_m$ be the 
coordinates of the $n$-tuple $(f_1,f_2,\ldots,f_n)$ in $W$,
let $\mu_1,\mu_2,\ldots,\mu_p$ be the coordinates 
of $h$ in $S_N$ where $p=\binom{n+N-1}{N}$,
and let $z_1,z_2,\ldots,z_q$ be the coordinates
of the $n$-tuple $(g_1,g_2,\ldots,g_n)$ in 
$S_{N-a_1}\oplus S_{N-a_2} \oplus \cdots \oplus S_{N-a_n}$ where $q=\sum_{i=1}^n\binom{n+N-a_i-1}{N-a_i}$.
Then equation (1) is equivalent to the following matrix equation:
\[
A
\begin{pmatrix} 
z_1\\ 
z_2 \\ 
\vdots\\ 
z_q
\end{pmatrix}
=
\begin{pmatrix}  
\mu_1\\ 
\mu_2 \\ 
\vdots\\ 
\mu_p
\end{pmatrix}, \tag{2}
\] 
where the $p\times q$ matrix $A$ is determined by equation (1) 
and the entries of $A$ are either some $\lambda_i$ or zero. 
Thus, $f_1,f_2,\ldots,f_n$ is a regular sequence if and only if 
given any vector $(\mu_1,\mu_2,\ldots,\mu_p)^T$, there exists a 
vector $(z_1,z_2,\ldots,z_q)^T$ satisfying equation (2), 
or equivalently, $\mbox{rank}(A)=p$. So, the $n$-tuple 
$(f_1,f_2,\ldots,f_n)$ is in $\Sigma$ if and only if 
the $m$-tuple $(\lambda_1,\lambda_2,\ldots,\lambda_m)$ 
is in $k^m-V(I_p(A))$. 
Since $k^m-V(I_p(A))$ is a Zariski-open subset of $k^m$,
it follows that $\Sigma$ is a Zariski-open subset of $W$.
Note that $x_1^{a_1},x_2^{a_2},\ldots, x_n^{a_n}$ is a regular
sequence in $S$. Therefore, $\Sigma$ is a non-empty Zariski-open subset of $W$.

If the field $k$ is infinite, then the affine space $W$
is irreducible, so that the non-empty Zariski-open subset 
$\Sigma$ is dense in $W$, which implies that generically, a 
sequence of homogeneous polynomials $f_1,f_2,\ldots, f_n$ 
of degrees $a_1,a_2,\ldots,a_n$ is a regular sequence in $S$.
\end{proof}

To illustrate equation (2) in the above proof of Proposition \ref{P:generic},
we have the following example.

\begin{example}
Let $S=k[x_1,x_2]$ and $a_1=a_2=2$, then the numbers 
defined in the above proof are $m=6$, $N=3$, $p=q=4$.
Let
\begin{align*}
f_1&=\lambda_1x_1^2+\lambda_2x_1x_2+\lambda_3x_2^2,\ \quad \ g_1=z_1x_1+z_2x_2,\\
f_2&=\lambda_4x_1^2+\lambda_5x_1x_2+\lambda_6x_2^2,\ \quad \ g_2=z_3x_1+z_4x_2,\\
h  &=\mu_1x_1^3+\mu_2x_1^2x_2+\mu_3x_1x_2^2+\mu_4x_2^3.
\end{align*}
Then $f_1g_1+f_2g_2=h$ is equivalent to the following matrix equation:
\[
A\begin{pmatrix} 
z_1\\ 
z_2 \\ 
z_3\\ 
z_4
\end{pmatrix}
=
\begin{pmatrix} 
\lambda_1&0 &\lambda_4 &0 \\ 
\lambda_2 &\lambda_1 &\lambda_5 &\lambda_4 \\ 
\lambda_3 &\lambda_2 &\lambda_6 &\lambda_5\\ 
0 &\lambda_3 &0 &\lambda_6
\end{pmatrix}
\begin{pmatrix} 
z_1\\ 
z_2 \\ 
z_3\\ 
z_4
\end{pmatrix}
=
\begin{pmatrix}  
\mu_1\\ 
\mu_2 \\ 
\mu_3\\ 
\mu_4
\end{pmatrix}.
\]
Since $z_1$ in $g_1$ can only be multiplied by $f_1$,
it follows that only the coefficients $\lambda_1,\lambda_2,\lambda_3$ 
of $f_1$ can appear in the first column of the matrix $A$ and each coefficient 
of $f_1$ appear only once in the first column. The other columns of $A$ 
have similar properties. Also, one can see that $f_1,f_2$ 
is a regular sequence in $S$ if and only if $\det(A)\neq 0$.

\end{example}

\begin{remark}\label{R:A}
Illustrated by the above example, we have that in general, the matrix $A$ in the 
proof of Proposition \ref{P:generic} has the following property:
for any column of $A$, there exists some $f_i$ such that 
the entries of the column are either the coefficients of $f_i$ or zero, and each 
coefficient of $f_i$ appears only once in that column. 
This observation about the matrix $A$ will be used in 
the proof of Theorem \ref{T:close}, which is the main result of this section.

If $k=\mathbb{C}$, then by Theorem 2.3 in Chapter 3 of \cite{B:CLO},
we see that $f_1,f_2,\ldots,f_n$ is a regular sequence in $S$
if and only if the resultant $\mbox{Res}(f_1,f_2,\ldots,f_n)$ is nonzero.
If $n=2$ or $a_1=a_2=\cdots=a_n=1$, then $\mbox{Res}(f_1,f_2,\ldots,f_n)=\det(A)$. 
For other cases, the matrix $A$ is not even a square matrix. And although
we have that $V(I_p(A))=V(\mbox{Res}(f_1,f_2,\ldots,f_n))$, it is not clear if
$\mbox{Res}(f_1,f_2,\ldots,f_n)$ can always be expressed as a single determinant. 
\end{remark}

In order to state Theorem \ref{T:close}, we need the following definition.
In the rest of this section, we will assume $k=\mathbb{C}$.

\begin{definition}\label{D:distance} 
Let $S=\mathbb{C}[x_1,x_2,\ldots,x_n]$ and $f,g$ be two
homogeneous polynomials in $S$ of degree $d$.
Suppose
\begin{align*}
f&=\lambda_1x_1^d+\lambda_2x_1^{d-1}x_2+\cdots+\lambda_rx_n^d,\\
g&=\nu_1x_1^d+\nu_2x_1^{d-1}x_2+\cdots+\nu_rx_n^d,
\end{align*}
where $r=\binom{n+d-1}{d}$, then we define
\[
\mbox{d}(f,g)=|\lambda_1-\nu_1|+|\lambda_2-\nu_2|+\cdots+|\lambda_r-\nu_r|,
\]
and we call it \emph{the distance between $f$ and $g$}.
\end{definition}

Given $1\leq a_1\leq a_2\leq \cdots \leq a_n$, let 
$W=S_{a_1}\oplus S_{a_2}\oplus \cdots \oplus S_{a_n}$. 
For any two $n$-tuples $(f_1,f_2,\ldots,f_n)$ 
and $(g_1,g_2,\ldots,g_n)$ in $W$, we can define the distance 
between them as follows:
\[
\mbox{d}\left((f_1,f_2,\ldots,f_n),(g_1,g_2,\ldots,g_n)\right)=\mbox{d}(f_1,g_1)+\mbox{d}(f_2,g_2)+\cdots+\mbox{d}(f_n,g_n).
\]
This definition makes $W$ into a metric space, whose induced 
topology is larger than the Zariski topology on $W$. 
So by Proposition \ref{P:generic}, the set $\Sigma$ 
is a non-empty open subset in the metric space $W$. 
Therefore, given any regular sequence $f_1,f_2,\ldots,f_n$
of homogeneous polynomials with degrees $a_1,a_2,\ldots,a_n$ 
in $S=\mathbb{C}[x_1,x_2,\ldots,x_n]$, there exists $\varepsilon>0$
such that for any homogeneous polynomials $g_1,g_2,\ldots,g_n\in S$
of degrees $a_1,a_2,\ldots,a_n$, if $\mbox{d}(g_1,f_1)<\varepsilon$,
$\mbox{d}(g_2,f_2)<\varepsilon$, $\ldots$, $\mbox{d}(g_n,f_n)<\varepsilon$,
then $g_1,g_2,\ldots,g_n$ is also a regular sequence in $S$. 

The question is: if $f_1=x_1^{a_1},f_2=x_2^{a_2},\ldots, f_n=x_n^{a_n}$,
what is the maximal value for $\varepsilon$? The answer is 
$\varepsilon=1$, as is shown in the following theorem.

\begin{theorem}\label{T:close}
Let $f_1,f_2,\ldots,f_n$ be a sequence of homogeneous polynomials
of degrees $1\leq a_1\leq a_2\leq \cdots \leq a_n$ in $S=\mathbb{C}[x_1,x_2,\ldots,x_n]$.
If $\mbox{d}(f_i,x_i^{a_i})<1$ for $i=1,2,\ldots,n$, then 
$f_1,f_2,\ldots, f_n$ is a regular sequence in $S$. 
\end{theorem}

To prove this theorem, we need the following lemma.

\begin{lemma}\label{L:det}
Let $A=(a_{ij})_{n\times n}\in M_n(\mathbb{C})$. 
If for every $j=1,2,\ldots,n$, we have
\[
|a_{jj}|>|a_{1j}|+\cdots+|a_{j-1,j}|+|a_{j+1,j}|+\cdots+|a_{nj}|,
\]
then $\det(A)\neq 0$.
\end{lemma}

\begin{proof}
Suppose, for contradiction, that $\det(A)=0$.
Let $\alpha_i=(a_{i1},a_{i2},\ldots,a_{in})$ for $i=1,2,\ldots,n$.
Then $\alpha_1,\alpha_2,\ldots,\alpha_n$ are linearly dependent.
Hence, there exist $l_1,l_2,\ldots,l_n \in \mathbb{C}$,
which are not all zero, such that 
\[
l_1\alpha_1+l_2\alpha_2+\cdots+l_n\alpha_n=0.
\]
Let $|l_j|=\max\{|l_i||i=1,2,\ldots,n\}>0$,
then
\[
\alpha_j=-\frac{l_1}{l_j}\alpha_1-\cdots-\frac{l_{j-1}}{l_j}\alpha_{j-1}
         -\frac{l_{j+1}}{l_j}\alpha_{j+1}-\frac{l_n}{l_j}\alpha_n.
\]
Hence,
\[
a_{jj}=-\frac{l_1}{l_j}a_{1j}-\cdots-\frac{l_{j-1}}{l_j}a_{j-1,j}
         -\frac{l_{j+1}}{l_j}a_{j+1,j}-\frac{l_n}{l_j}a_{nj}.
\]
Thus,
\begin{align*}
|a_{jj}|&\leq\left|\frac{l_1}{l_j}\right||a_{1j}|+\cdots+\left|\frac{l_{j-1}}{l_j}\right||a_{j-1,j}|
         +\left|\frac{l_{j+1}}{l_j}\right||a_{j+1,j}|+\left|\frac{l_n}{l_j}\right||a_{nj}|\\
        &\leq |a_{1j}|+\cdots+|a_{j-1,j}|+|a_{j+1,j}|+|a_{nj}|,
\end{align*}
which contradicts to the assumption. Therefore, $\det(A)\neq 0$.
\end{proof}

\begin{proof}[Proof of Theorem \ref{T:close}]
For every $i=1,2,\ldots,n$, let $\nu_i$ be the coefficient 
of $x_i^{a_i}$ in $f_i$ and $c_i$ the sum of the absolute value 
of the other coefficients in $f_i$. Since $\mbox{d}(f_i,x_i^{a_i})<1$,
it follows that $c_i+|\nu_i-1|<1$, which implies that 
$c_i<1-|\nu_i-1|\leq |1+(\nu_i-1)|=|\nu_i|$.

Let $A$ be the $p\times q$ matrix defined in the proof of 
Proposition \ref{P:generic}. Since $x_1^{a_1},x_2^{a_2},\ldots,x_n^{a_n}$ 
is a regular sequence in $S$, it follows that there is a 
$p\times p$ submatrix $B$ of $A$, which is invertible 
when evaluated at the $n$-tuple $(x_1^{a_1},x_2^{a_2},\ldots,x_n^{a_n})$.
Let $B(x_1^{a_1},x_2^{a_2},\ldots,x_n^{a_n})$ denote the 
matrix $B$ evaluated at $(x_1^{a_1},x_2^{a_2},\ldots,x_n^{a_n})$.
By Remark \ref{R:A}, it is easy to see that 
the column vectors of $B(x_1^{a_1},x_2^{a_2},\ldots,x_n^{a_n})$
are of the form $(0,\ldots,0,1,0,\ldots,0)^T$.
Since $B(x_1^{a_1},x_2^{a_2},\ldots,x_n^{a_n})$ is invertible,
by rearranging the column vectors of $B$, 
we can assume that $B(x_1^{a_1},x_2^{a_2},\ldots,x_n^{a_n})$
is the identity matrix. 

Let $B(f_1,f_2,\ldots,f_n)$ be the 
matrix $B$ evaluated at the $n$-tuple $(f_1 ,f_2 ,\ldots, f_n)$.
For every $i=1,2,\ldots,p$, let $\eta_i$ be the $i$-th column vector
of $B(f_1,f_2,\ldots,f_n)$. Then by Remark \ref{R:A}, it is easy to 
see that the $i$-th entry in $\eta_i$ is $\nu_j$ for some $1\leq j\leq n$,
and all the other coefficients of $f_j$ appear as the entries of $\eta_i$
exactly once, and the remaining entries of $\eta_i$ are zeros.
Since $c_j<|\nu_j|$, it follows that the matrix $B(f_1,f_2,\ldots,f_n)$
satisfies the assumptions of Lemma \ref{L:det}, and then
$\det(B(f_1,f_2,\ldots,f_n))\neq 0$. So the matrix $A$ evaluated at 
the $n$-tuple $(f_1,f_2,\ldots,f_n)$ has rank $p$. 
Therefore, $f_1,f_2,\ldots,f_n$ is a regular sequence in $S$.
\end{proof}

\begin{remark}\label{R:close}
It is easy to see that in $S=\mathbb{C}[x_1,x_2]$, if $\mbox{d}(f_1,x_1^{a_1})=1$ 
and $\mbox{d}(f_2,x_2^{a_2})=1$, then $f_1,f_2$ may not be a regular sequence in $S$.
For example, $f_1=x^2+xy$, $f_2=xy^2+y^3$ is not a regular sequence in $S$.
\end{remark}

The regular sequences obtained in Theorem \ref{T:close} are very different
from the regular sequences obtained in Theorem \ref{T:pst}. 
It would be interesting to know if the EGH Conjecture holds for these two 
classes of regular sequences. 
Besides, the above two classes of regular sequences 
are in $S=\mathbb{C}[x_1,x_2,\ldots,x_n]$. 
If $S=\mathbb{Z}/2\mathbb{Z}[x_1,x_2,\ldots,x_n]$,
then the sequence in Theorem \ref{T:pst} is not a 
regular sequence because the polynomial system has 
a non-zero solution $(1,1,0,\ldots,0)$. 
It would be interesting to find a new class of 
regular sequences in $S=\mathbb{Z}/2\mathbb{Z}[x_1,x_2,\ldots,x_n]$.

We end this paper with the following proposition. We put it 
in this paper because its proof is similar to the proof of Proposition \ref{P:generic}. 

\begin{proposition}\label{P:nbhd}
Let $S=\mathbb{C}[x_1,x_2,\ldots,x_n]$. Let $f_1,f_2,\ldots,f_t$ be 
$\mathbb{C}$-linearly independent polynomials in $S_i$ with $i\geq 1$.
Then there exists $\varepsilon>0$ such that for any homogeneous 
polynomials $g_1,g_2,\ldots,g_t$ in $S_i$ satisfying 
$\mbox{d}(g_1,f_1)<\varepsilon$,$\mbox{d}(g_2,f_2)<\varepsilon$,$\ldots$,$\mbox{d}(g_t,f_t)<\varepsilon$,
we have that $g_1,g_2,\ldots,g_t$ are $\mathbb{C}$-linearly independent and
\[
\dim_{\mathbb{C}}(f_1,f_2,\ldots,f_t)_{i+1}\leq \dim_{\mathbb{C}}(g_1,g_2,\ldots,g_t)_{i+1}.
\]
\end{proposition}

\begin{proof}
It is easy to see that  $g_1,g_2,\ldots,g_t$ are $\mathbb{C}$-linearly independent 
when $\varepsilon$ is sufficiently small. Let
\[
\Sigma=\{(g_1,g_2,\ldots,g_t)\in (S_i)^t|\dim_{\mathbb{C}}(g_1,g_2,\ldots,g_t)_{i+1}\geq \dim_{\mathbb{C}}(f_1,f_2,\ldots,f_t)_{i+1}\}.
\]
Since the $t$-tuple $(f_1,f_2,\ldots,f_t)$ is in $\Sigma$, it follows that $\Sigma$ is 
non-empty, so that it suffices to prove that $\sigma$ is an open
subset of the affine $\mathbb{C}$-space $(S_i)^t$. Note that 
under the monomial basis as defined in the proof of Proposition \ref{P:generic}, 
$(S_i)^t$ is isomorphic to $\mathbb{C}^N$ where $\textstyle{N=t\binom{n+i-1}{i}}$.

For any $t$-tuple $(g_1,g_2,\ldots,g_t)\in (S_i)^t$, $(g_1,g_2,\ldots,g_t)_{i+1}$
is the $\mathbb{C}$-vector space spanned by the $nt$ polynomials $x_1g_1,x_2g_1,\ldots,x_ng_t$.
Let $B$ be the $nt\times \textstyle{\binom{n+i}{i+1}}$ matrix with each row 
being the coordinates of some $x_lg_j$ under the monomial basis of $S_{i+1}$,
where $l=1,2,\ldots,n$ and $j=1,2,\ldots,t$. 
If the coordinates of $t$-tuple $(g_1,g_2,\ldots,g_t)$ under the monomial 
basis of $(S_i)^t$ is $(\lambda_1,\lambda_2,\ldots,\lambda_N)$, then 
the entries of $B$ are either some $\lambda_i$ or zero. 
Let $p=\dim_{\mathbb{C}}(f_1,f_2,\ldots,f_t)_{i+1}$, 
then it is easy to see that the $t$-tuple $(g_1,g_2,\ldots,g_t)$ is in $\Sigma$
if and only if the $N$-tuple $(\lambda_1,\lambda_2,\ldots,\lambda_N)$ is in 
$\mathbb{C}^N-V(I_p(B))$.
Since $\mathbb{C}^N-V(I_p(B))$ is an open subset of $\mathbb{C}^N$,
it follows that $\Sigma$ is an open subset of $(S_i)^t$.
\end{proof}

\end{document}